\documentclass[reqno,tbtags,intlimits,a4paper,oneside,12pt]{amsart}
\usepackage[cp1251]{inputenc}%
\usepackage[english]{babel}
\usepackage{amssymb,upref,calc,url}
\usepackage[mathscr]{eucal}
\usepackage{graphicx}
\usepackage{amsmath,amscd,amsthm,verbatim}
\usepackage[all]{xy}
\usepackage[in]{fullpage}
\newtheorem{thm}{Theorem}[section]
\newtheorem{lem}[thm]{Lemma}

\newtheorem{prb}[thm]{Problem}

\theoremstyle{definition}

\DeclareMathOperator{\Der}{Der}

\begin{document}

\title{The Problem of Differentiation\\
of Hyperelliptic Functions}
\author{Elena~Yu.~Bunkova}
\address{Steklov Mathematical Institute of Russian Academy of Sciences, 8 Gubkina St., Moscow, 119991, Russia.}
\email{bunkova@mi.ras.ru}
\thanks{Supported in part by Young Russian Mathematics award and the RFBR project 17-01-00366 A}

\begin{abstract}
In this work we describe a construction that leads to an explicit solution of the problem of differentiation of hyperelliptic functions.
A classical genus $g=1$ example of such a solution is the result of F.~G.~Frobenius and L.~Stickelberger \cite{FS}.

Our method follows the works \cite{B2} and \cite{B3} that led to constructions of explicit solutions of the problem for genus $g=2$ and $g=3$.
\end{abstract}

\maketitle

\section{Introduction} \label{s1}

We consider meromorphic functions $f$ in $\mathbb{C}^g$. A vector $\omega \in \mathbb{C}^g$ is called a period
for~$f$ if $f(z+\omega) = f(z)$ for any~$z \in \mathbb{C}^g$.
If the periods of $f$ form a lattice $\Gamma$ of rank~$2g$ in~$\mathbb{C}^g$, then $f$ is called an \emph{Abelian function}. We say that an Abelian function is a meromorphic function on the complex torus $T^g = \mathbb{C}^g/\Gamma$. We denote the coordinates in $\mathbb{C}^g$ by $(z_1, z_3, \ldots, z_{2g-1})$.

Let us consider hyperelliptic curves of genus $g$ in the model
\begin{equation} \label{1}
\mathcal{V}_\lambda = \{(X,Y)\in\mathbb{C}^2 \colon
Y^2 = X^{2g+1} + \lambda_4 X^{2 g - 1}  + \lambda_6 X^{2 g - 2} + \ldots + \lambda_{4 g} X + \lambda_{4 g + 2}\}. 
\end{equation}
Such a curve depends on the parameters $\lambda = (\lambda_4, \lambda_6, \ldots, \lambda_{4 g}, \lambda_{4 g + 2}) \in \mathbb{C}^{2 g}$.

Denote by $\mathcal{B} \subset \mathbb{C}^{2g}$ the subspace of parameters such that $\mathcal{V}_{\lambda}$ is non-singular for $\lambda \in \mathcal{B}$.
We have $\mathcal{B} = \mathbb{C}^{2g} \backslash \Sigma$ where $\Sigma$ is the discriminant curve.

A \emph{hyperelliptic function of genus} $g$ (see \cite{B2, BEL18}) is 
a meromorphic function in $\mathbb{C}^g \times \mathcal{B}$,
such that for each $\lambda \in \mathcal{B}$ it's restriction to $\mathbb{C}^g \times \lambda$
is Abelian with $T^g$ the Jacobian $\mathcal{J}_\lambda$ of~$\mathcal{V}_\lambda$.
We denote the field of hyperelliptic functions of genus $g$ by $\mathcal{F}$. See \cite{BEL18} for it's properties. 

Let $\mathcal{U}$ be the space of the fiber bundle $\pi: \mathcal{U} \to \mathcal{B}$ with fiber over $\lambda \in \mathcal{B}$ the Jacobian~$\mathcal{J}_\lambda$
of the curve $\mathcal{V}_\lambda$.
Thus, a hyperelliptic function is a meromorphic function in~$\mathcal{U}$.
According to Dubrovin--Novikov theorem \cite{DN}, there is a birational isomorphism between $\mathcal{U}$ and the complex linear space $\mathbb{C}^{3g}$.


\begin{prb}[\cite{BEL18}] \label{p1} 
For each~$g$ describe the Lie algebra $\Der \mathcal{F}$ of differentiations of $\mathcal{F}$,
that is find~$3 g$ independent differential operators $\mathscr{L}$ such that $\mathscr{L} \mathcal{F} \subset \mathcal{F}$.
\end{prb}

In case $g=1$ the solution of this problem is classical~\cite{FS}.
A method for solving this problem in a general case was presented in \cite{BL0, BL}. A good overview of this approach is given in \cite{BEL18}. It turned out that it is hard to follow this method to obtain explicit answers. 

Explicit solutions to this problem for $g=2$ and $g=3$ were first found in \cite{B2} and \cite{B3}. This works allow us to present a general method that is useful for any genus. Here we describe the general construction of this method.

We use the theory of hyperelliptic Kleinian functions (see \cite{Baker, BEL, BEL-97, BEL-12}, and~\cite{WW} for elliptic functions).
Take the coordinates
$(z, \lambda) = (z_1, z_3, \ldots, z_{2 g -1}, \lambda_4, \lambda_6, \ldots, \lambda_{4 g}, \lambda_{4 g + 2})$
in $\mathbb{C}^g \times \mathcal{B} \subset \mathbb{C}^{3g}$.
Let $\sigma(z, \lambda)$ be the hyperelliptic sigma function (or elliptic sigma function in genus $g=1$ case). We denote $\partial_k = {\partial \over \partial z_k}$.
Following \cite{B2, B3, BEL18}, we use the notation
\begin{equation} \label{n}
\zeta_{k} = \partial_k \ln \sigma(z, \lambda), \qquad
\wp_{i; k_1, \ldots, k_n} = - \partial_1^{i} \partial_{k_1} \cdots \partial_{k_n} \ln \sigma(z, \lambda),
\end{equation}
where $n \geqslant 0$, $i + n \geqslant 2$, $k_s \in \{ 1, 3, \ldots, 2 g - 1\}$. In the case $n = 0$ we will skip the semicolon.
Note that our notation for the variables $z_k$ differs from the one in \cite{BEL, BEL-97, BEL-12} as $u_{i} = z_{2g + 1 - 2 i}$.
The functions $\wp_{i; k_1, \ldots, k_n}$ provide us with examples of hyperelliptic functions. 

A key to our approach to the problem is the following theorem: 
\begin{thm}[\cite{BEL}] \label{TBEL} For $i, k \in \{1, 3, \ldots, 2 g - 1\}$ we have the relations
\begin{align}
\wp_{3; i} = & 6 \wp_2 \wp_{1; i} + 6 \wp_{1; i+2} - 2 \wp_{0;3, i} + 2 \lambda_{4} \delta_{i,1}, \label{2}\\
\wp_{2; i} \wp_{2; k} = & 4 \left(\wp_2 \wp_{1; i} \wp_{1; k} + \wp_{1; k} \wp_{1; i+2} + \wp_{1; i} \wp_{1; k+2} + \wp_{0; k+2, i+2}\right)  
   - \nonumber \\
   & - 2 (\wp_{1; i} \wp_{0;3,k} + \wp_{1; k} \wp_{0;3,i} + \wp_{0; k, i+4} + \wp_{0;i, k+4}) + \label{3} \\
   &+ 2 \lambda_4 (\delta_{i,1} \wp_{1; k} + \delta_{k,1} \wp_{1;i}) + 2 \lambda_{i+k+4} (2 \delta_{i,k} + \delta_{k, i-2} + \delta_{i, k-2}). \nonumber
\end{align}
\end{thm}
\begin{proof}
 In \cite{BEL} we have formulas (4.1) and (4.8). Using the notation \eqref{n} we get \eqref{2} from~(4.1) and \eqref{3} from (4.8).
\end{proof}

\section{The problem for polynomial vector fields}

The work \cite{BPol} constructs the theory of polynomial Lie algebras. Here we describe its connection with Problem \ref{p1}.

Consider the complex space $\mathbb{C}^{3g}$ with coordinates $x = (x_{i,j})$, 
where $i \in \{ 1,2,3 \}$, \mbox{$j \in \{1, 3, \ldots, 2 g -1\}$.}
We define the map $\varphi: \mathcal{U} \dashrightarrow \mathbb{C}^{3g}$
by
\[
\varphi: (z, \lambda) \mapsto (x_{i,j}) = (\wp_{i;j}(z, \lambda)).
\]
This map has the following property, proposed by V.~M.~Buchstaber (see \cite{B2}):

\begin{thm} \label{This}
The functions $\varphi^*(x_{i,j})$ give a set of generators of $\mathcal{F}$.
\end{thm}

\begin{proof}
We show that the functions $\wp_{i;j}(z, \lambda)$,
where $i \in \{ 1,2,3 \}$, \mbox{$j \in \{1, 3, \ldots, 2 g -1\}$,}
give a set of generators of $\mathcal{F}$.

We use a fundamental result from the theory of hyperelliptic Abelian functions (see~\cite[Chapter 5]{BEL-12}):
Any hyperelliptic function can be presented as a rational function
in $\wp_{1;k}$ and $\wp_{2;k}$, where $k \in \{1, 3, \ldots, 2 g - 1\}$.
Theorem \ref{TBEL} gives a set of relations between the derivatives of this functions.

Now by \cite[Corollary 5.2]{B3}, the functions
$(\varphi^*(x_{i,j}), \varphi^*(w_{k,l}), \varphi^*(\lambda_s))$ in the notation of this Corollary give a set of generators of $\mathcal{F}$. By \cite[Theorem 5.3]{B3} we obtain the claim of Theorem \ref{This}.
\end{proof}

Another property of $\varphi$ follows form \cite[Corollary 5.5]{B3}.
For each $g$ a there is polynomial map $p: \mathbb{C}^{3g} \to \mathbb{C}^{2g}$, such that we get the diagram
\begin{equation} 
 \xymatrix{
	\mathcal{U} \ar[d]^{\pi} \ar@{-->}[r]^{\varphi} & \mathbb{C}^{3 g} \ar[d]^{p}\\
	\mathcal{B} \ar@{^{(}->}[r] & \mathbb{C}^{2g}\\
	}   \label{d}
\end{equation}
Here $\mathcal{B} \subset \mathbb{C}^{2g}$ is the inclusion like in section \ref{s1}, with coordinates $\lambda$ in $\mathbb{C}^{2 g}$. 

We note that the proof of \cite[Theorem 5.3]{B3} gives a construction to obtain the polynomial maps $p$ explicitly. Examples of this maps for $g = 1,2,3$ are given in \cite{B3}. The work \cite[Theorem 3.2]{BEL18} claims that this polynomial maps are of degree at most $3$.

We refer the reader to \cite{BPol} for the theory of polynomial Lie algebras.
Denote the ring of polynomials in $\lambda \in \mathbb{C}^{2g}$ by $\mathcal{P}$.
Let us consider the polynomial map~$p\colon \mathbb{C}^{3g} \to \mathbb{C}^{2g}$.
A vector field $\mathcal{L}$ in $\mathbb{C}^{3g}$ will be called projectable for $p$ if there exists a vector field $L$ in $\mathbb{C}^{2g}$ such that
\[
 \mathcal{L}(p^* f) = p^* L(f) \quad \text{for any} \quad f \in \mathcal{P}.
\]
The vector field $L$ will be called the pushforward of $\mathcal{L}$.
A corollary of this definition is that for a projectable vector field $\mathcal{L}$ we have $\mathcal{L}(p^* \mathcal{P}) \subset p^* \mathcal{P}$.

\begin{prb}[\text{\cite[Problem 6.1]{B3}}] \label{p6}
Find $3 g$ polynomial vector fields in $\mathbb{C}^{3g}$ projectable for $p\colon \mathbb{C}^{3g} \to \mathbb{C}^{2g}$ and
independent at any point in $p^{-1}(\mathcal{B})$.
Construct their polynomial Lie algebra.
\end{prb}

The connection of this problem to Problem \ref{p1} is straightforward.
Given a solution to Problem \ref{p6} for each of the $3g$ vector fields~$\mathcal{L}_k$
with pushforwards $L_k$ we will restore the vector fields~$\mathscr{L}_k$ 
projectable for $\pi$ with pushforwards $L_k$ and such that $\mathscr{L}_k (\varphi^* x_{i,j}) = \varphi^* \mathcal{L}_k (x_{i,j})$
for the coordinate functions $x_{i,j}$ in $\mathbb{C}^{3g}$.
As $\varphi^* x_{i,j}$ are the generators of $\mathcal{F}$ and $\mathcal{L}_k (x_{i,j})$ is a polynomial in $x_{i,j}$,
this gives $\mathscr{L}_k (\varphi^* x_{i,j}) \in \mathcal{F}$ and $\mathscr{L}_k \in \Der \mathcal{F}$.

The plan to solve Problem \ref{p6} is the following. For each $g$:
\begin{itemize}
 \item Find the ``odd polynomial vector fields'', i.e. the $g$ independant polynomial vector fields $\mathcal{L}_{1}, \mathcal{L}_{3}, \ldots, \mathcal{L}_{2 g - 1}$ projectable for $p$ with zero pushforward.
 \item Define $2g$ independant polynomial vector fields $L_{0}, L_{2}, L_{4}, \ldots, L_{4 g - 2}$ in $\mathcal{B}$.
 \item Find the ``even polynomial vector fields'', i.e. the $2g$ polynomial vector fields $\mathcal{L}_{0}, \mathcal{L}_{2}, \mathcal{L}_{4}, \ldots, \mathcal{L}_{4 g - 2}$
 projectable for $p$ with pushforwards $L_{0}, L_{2}, L_{4}, \ldots, L_{4 g - 2}$ .
 \item Construct their polynomial Lie algebra.
\end{itemize}

We will do this steps in the following sections. In the last section we give the explicit solutions for problem \ref{p1} that can be constucted by this method (see \cite{B3}).

\section{Odd polynomial vector fields}

\begin{lem}[{\cite[Lemma 6.2 and Lemma 6.3]{B3}}] \label{Lem1}
 We have 
\begin{align}
\mathcal{L}_1 &= \sum_j x_{2,j} {\partial \over \partial x_{1,j}} + x_{3, j} {\partial \over \partial x_{2,j}} + 4 (2 x_{2} x_{2,j} + x_{3} x_{1, j} + x_{2, j+2}) {\partial \over \partial x_{3,j}}
\label{L1}
\end{align}
where $x_{2,2 g + 1} = 0$. 
 For $s = 3, 5, \ldots, 2g - 1$ we have 
\begin{multline}
\mathcal{L}_{s} = x_{2,s} {\partial \over \partial x_{2}} + x_{3,s} {\partial \over \partial x_{3}} \label{Ls}
+ \mathcal{L}_1(x_{3,s}) {\partial \over \partial x_{4}} + \\ +
\sum_{k=1}^{g-1} z_{1,s,2k+1} {\partial \over \partial x_{1,2k+1}} + \mathcal{L}_1(z_{1,s,2k+1}) {\partial \over \partial x_{2,2k+1}}
+ \mathcal{L}_1(\mathcal{L}_1(z_{1,s,2k+1})) {\partial \over \partial x_{3,2k+1}}.
\end{multline}
for some $y_{1,s,2k+1} = \mathcal{L}_s(x_{1,2k+1})$.
\end{lem}

This lemma determines the odd polynomial vector fields given the value
$\mathcal{L}_s(x_{1,2k+1})$. For this value we use the construction of Korteweg--de Vries hierarchy \cite[Section~4.4]{BEL-97}.

\eject

The Korteweg--de Vries equation 
\[
 u_t = 6 u u_x - u_{xxx}
\]
for $x = z_1$, $-4 t = z_3$, $\Phi_2 = {1 \over 2} u$, $\Phi_4 = - {3 \over 2} \Phi_2^2 + {1 \over 4} \partial_1 \Phi_2$ takes the form
\[
 \partial_3 \Phi_2 = \partial_1 \Phi_4.
\]
It is the first equation of the Korteweg--de Vries hierarchy, which is an infinite system of differential equations 
\[
 \partial_{2k-1} \Phi_2 = \partial_1 \Phi_{2k}, \quad k = 2,3,4,\ldots
\]
where 
\[
 \partial_1 \Phi_{2k+2} = \mathcal{R} \partial_1 \Phi_{2k}
 \quad \text{ and } \quad
 \mathcal{R} = {1 \over 4} \partial_1^2 - 2 \Phi_2 - \Phi_2' \partial_1^{-1}
\]

\begin{thm}[{\cite[Theorem~4.12]{BEL-97}}]
 The function $u = 2 \wp_{2}(z)$ is a $g$-gap solution of the~Korteweg--de Vries system.
\end{thm}

This gives us a system of equations
\[
 \mathcal{L}_s(x_2) = \mathcal{L}_1 \Phi_s(x_2)
\]
with differential polynomials $\Phi_s$.
Thus in Lemma \ref{Lem1} we have
\[
 \mathcal{L}_s (x_{1,2k+1}) = \mathcal{L}_s ( \Phi_{2k}(x_2)).
\]
This determines $y_{1,s,2k+1}$.

\section{Even polynomial vector fields}

First we define the polynomial vector fields in $\mathcal{B}$. Recall
$\mathcal{B} = \mathbb{C}^{2g} \backslash \Sigma$ where $\Sigma$ is the discriminant curve.

For the vector fields $L_{0}, L_{2}, L_{4}, \ldots, L_{4 g - 2}$ in $\mathcal{B}$ we take the vector fields tangent to $\Sigma$, that are obtained from the convolution of invariants of the group $A_\mu$, see the construction by D.B.Fuchs in~\cite[Section 4]{A}. See also~\cite{BPol} 
and~\cite{BMinf}.

We consider $\mathbb{C}^{2g}$ with coordinates $(\lambda_4, \lambda_6, \ldots, \lambda_{4 g}, \lambda_{4 g + 2})$ and set $\lambda_s = 0$
for every $s \notin \{ 4,6, \ldots, 4 g, 4 g + 2\}$.
For $k, m \in \{ 1, 2, \ldots, 2 g\}$, $k\leqslant m$ set
\[
 T_{2k, 2m} = 2 (k+m) \lambda_{2k+2m} + \sum_{s=2}^{k-1} 2 (k + m - 2 s) \lambda_{2s} \lambda_{2k+2m-2s}
 - {2 k (2 g - m + 1) \over 2 g + 1} \lambda_{2k} \lambda_{2m},
\] 
and for $k > m $ set $T_{2k, 2m} = T_{2m, 2k}$.
For $k = 0, 1, 2, \ldots, {2 g - 1}$ we
have the vector fields 
\begin{equation} \label{Lk}
 L_{2k} = \sum_{s = 2}^{2 g + 1} T_{2k + 2, 2 s - 2} {\partial \over \partial \lambda_{2s}}.
\end{equation}

 The expressions \eqref{Lk} give polynomial vector fields tangent to the discriminant curve.

Now we need to find polynomial vector fields $\mathcal{L}_{2k}$ projectable for $p$ with pushforwards~$L_{2k}$.
The vector field $\mathcal{L}_0$ is the Euler vector field on $\mathbb{C}^{3g}$, we have
\begin{align}
\mathcal{L}_0 &= \sum_j (j+1) x_{1,j} {\partial \over \partial x_{1,j}} + (j+2) x_{2,j} {\partial \over \partial x_{2,j}} +
(j+3) x_{3, j} {\partial \over \partial x_{3,j}}. \label{L0}
\end{align}

All the other vector fields are determined using the condition on the polynomial Lie algebra
\begin{equation} \label{com1}
  \begin{pmatrix}
 [\mathcal{L}_1, \mathcal{L}_0] \\
 [\mathcal{L}_1, \mathcal{L}_2] \\
 [\mathcal{L}_1, \mathcal{L}_4] \\
 [\mathcal{L}_1, \mathcal{L}_6] \\
 \vdots \\
 [\mathcal{L}_1, \mathcal{L}_{4 g - 4}] \\
 [\mathcal{L}_1, \mathcal{L}_{4 g - 2}] 
 \end{pmatrix}
= 
 \begin{pmatrix}
 -1 & 0 & 0 & \ldots & 0\\
 x_{1,1} & -1 & 0 & \ldots & 0\\
 x_{1,3} & x_{1,1} & -1 & \ldots & 0 \\
 x_{1,5} & x_{1,3} & x_{1,1} & \ldots & 0 \\
 \ddots &  \ddots &  \ddots &  \ddots &  \ddots \\
 0 & 0 & \ldots & x_{1, 2g-1} & x_{1, 2g-3}\\
 0 & 0 & \ldots & 0 & x_{1, 2g-1}
 \end{pmatrix}
 \begin{pmatrix}
  \mathcal{L}_1 \\ \mathcal{L}_3\\ \vdots \\ \mathcal{L}_{2g-1}
 \end{pmatrix}.
\end{equation}

A demonstration of this method for genus $g=4$ will follow in our upcoming works.

\section{Explicit solutions of the Problem of Differentiation of~Hyperelliptic~Functions} \label{s3}

\subsection{Genus 1} See~\cite{FS}. The generators of the $\mathcal{F}$-module $\Der \mathcal{F}$ are
\begin{align*}
\mathscr{L}_0 &= L_0 - z_1 \partial_{1}, & \mathscr{L}_1 &= \partial_{1},
& \mathscr{L}_2 &= L_2 - \zeta_1 \partial_{1},
\end{align*}
Their Lie algebra is
$
[\mathscr{L}_0, \mathscr{L}_1] = \mathscr{L}_1, [\mathscr{L}_0, \mathscr{L}_2] = 2 \mathscr{L}_2,
[\mathscr{L}_1, \mathscr{L}_2] = \wp_2 \mathscr{L}_1.
$

\subsection{Genus 2} \label{ss2} The generators of the $\mathcal{F}$-module $\Der \mathcal{F}$ are (see~\cite[Theorem 29]{B2}):
\begin{align*}
\mathscr{L}_0 &= L_0 - z_1 \partial_{1} - 3 z_3 \partial_{3}, &
\mathscr{L}_2 &= L_2 + \left(- \zeta_1 + {4 \over 5} \lambda_4 z_3\right) \partial_{1} - z_1 \partial_{3}, \\
\mathscr{L}_1 &= \partial_{1}, &
\mathscr{L}_4 &= L_4 + \left(- \zeta_3 + {6 \over 5} \lambda_6 z_3\right) \partial_{1} - \left(\zeta_1 + \lambda_4 z_3\right) \partial_{3}, \\
\mathscr{L}_3 &= \partial_{3}, &
\mathscr{L}_6 &= L_6 + {3 \over 5} \lambda_8 z_3 \partial_{1} - \zeta_3 \partial_{3}.
\end{align*} 
Their Lie algebra can be found in \cite[Theorem 32]{B2}.

\subsection{Genus 3} \label{ss3} The generators of the $\mathcal{F}$-module $\Der \mathcal{F}$ are (see~\cite[Theorem 10.1]{B3}):
\begin{align*}
\mathscr{L}_1 &= \partial_{1}, \qquad \mathscr{L}_3 = \partial_{3}, \qquad \mathscr{L}_5 = \partial_{5}, \\
\mathscr{L}_0 &= L_0 - z_1 \partial_{1} - 3 z_3 \partial_{3} - 5 z_5 \partial_{5}, \\
\mathscr{L}_2 &= L_2 - \left(\zeta_1 - {8 \over 7} \lambda_4 z_3 \right) \partial_{1}
- \left( z_1 - {4 \over 7} \lambda_4 z_5 \right) \partial_{3} - 3 z_3 \partial_{5}, \\
\mathscr{L}_4 &= L_4 - \left(\zeta_3 - {12 \over 7} \lambda_6 z_3 \right) \partial_{1}
- \left( \zeta_1 + \lambda_4 z_3 - {6 \over 7} \lambda_6 z_5 \right) \partial_{3} - (z_1 + 3 \lambda_4 z_5) \partial_{5}, \\
\mathscr{L}_6 &= L_6 - \left(\zeta_5 - {9 \over 7} \lambda_8 z_3 \right) \partial_{1}
- \left(\zeta_3 - {8 \over 7} \lambda_8 z_5 \right) \partial_{3}
- \left(\zeta_1 + \lambda_4 z_3 + 2 \lambda_6 z_5 \right) \partial_{5}, \\
\mathscr{L}_8 &= L_8 + \left({6 \over 7} \lambda_{10} z_3 - \lambda_{12} z_5\right) \partial_{1}
- \left(\zeta_5 - {10 \over 7} \lambda_{10} z_5 \right) \partial_{3}
- \left(\zeta_3 + \lambda_8 z_5 \right) \partial_{5}, \\
\mathscr{L}_{10} &= L_{10} + \left( {3 \over 7} \lambda_{12} z_3 - 2 \lambda_{14} z_5 \right) \partial_{1}
+ {5 \over 7} \lambda_{12} z_5 \partial_{3} - \zeta_5 \partial_{5}. 
\end{align*}
Their Lie algebra can be found in \cite[Corollary 10.2]{B3}.

\end{document}